\documentclass[11pt, leqno]{amsart}
\usepackage{amssymb, amsmath, amsthm, bbm, cite, cases, paralist}
\usepackage[margin=1in]{geometry}

\DeclareMathOperator*{\osc}{osc}

\numberwithin{equation}{section}

\theoremstyle{plain}
\newtheorem{theorem}{Theorem}[section]
\newtheorem{lemma}[theorem]{Lemma}
\newtheorem{proposition}[theorem]{Proposition}
\theoremstyle{remark}
\newtheorem{remark}[theorem]{Remark}
\theoremstyle{definition}
\newtheorem{defn}[theorem]{Definition}
\author{Yan Zhang}
\title{Asymptotic Behavior of a Nonlocal KPP Equation with an Almost Periodic Nonlinearity}
\subjclass[2010]{35B40, 35B27, 35K57, 35R09, 35D40}
\date{September 2, 2015}
\keywords{Nonlocal KPP equation, asymptotic behavior, homogenization}
\begin{document}
\begin{abstract}
We consider a space-inhomogeneous Kolmogorov-Petrovskii-Piskunov (KPP) equation with a nonlocal diffusion and an almost-periodic nonlinearity. By employing and adapting the theory of homogenization, we show that solutions of this equation asymptotically converge to its stationary states in regions of space separated by a front that is determined by a Hamilton-Jacobi variational inequality.
\end{abstract}
\maketitle
\section{Introduction}
\label{sec:intro}
\setcounter{section}{1}
The aim of this paper is to analyze the large space/long time asymptotic behavior of the nonlocal reaction-diffusion equation
\begin{equation} \label{originalequation}
u_t(x, t) - \int J(y)[u(x-y, t) - u(x, t)]dy  - f(x, u) = 0,
\end{equation}
where $J$ is a continuous, nonnegative, compactly supported, and symmetric kernel, and $f$ is a monostable (KPP) type nonlinearity in $u$ for which the canonical example is $f(u) = u(1-u)$. To study the asymptotic behavior of \eqref{originalequation}, we introduce the ``hyperbolic'' scaling $(x, t) \mapsto (\epsilon^{-1} x, \epsilon^{-1}t).$ As $\epsilon \rightarrow 0$, the time scaling reproduces long-time behavior of (\ref{originalequation}), while the space scaling reproduces in bounded sets behavior for large space variables. The new unknown is now given by $u^\epsilon(x, t) := u(\epsilon^{-1}x, \epsilon^{-1}t).$ We introduce an initial condition $u^\epsilon(\cdot, 0) = u_0(\cdot)$, and we can easily see that $u^\epsilon$ satisfies
\begin{equation} \label{uequation}
\left\{ \begin{array}{l}
\displaystyle u^\epsilon_t(x, t) - \frac{1}{\epsilon}\int J(y) [u^\epsilon(x-\epsilon y, t) dy - u^\epsilon(x, t)] dy - \frac{1}{\epsilon}f\left(\frac{x}{\epsilon}, u^\epsilon\right)=0  \text{ in } \mathbb{R}^n \times (0, \infty), \\
u^\epsilon(x, 0) = u_0(x).
\end{array}\right.
\end{equation}
The behavior of $u^\epsilon$ as $\epsilon \rightarrow 0$ is what we will consider to determine the asymptotic behavior of \eqref{originalequation}. To obtain a result concerning the convergence of $u^\epsilon$, it is necessary to make assumptions about the oscillatory behavior of $f$, and in this paper we assume that $f$ is an almost-periodic function in the $\frac{x}{\epsilon}$ variable. Our main result, Theorem \ref{maintheorem}, states that as $\epsilon \rightarrow 0$, $u^\epsilon$ respectively converges to the two equilibria of $f$, which for simplicity we take to be constant, in the two regions $\{\phi < 0\}$ and $\mathrm{int}(\{\phi = 0\})$, where $\phi$ is the solution of the Hamilton-Jacobi variational inequality
\begin{equation} \label{effectiveequation}
\left\{
\begin{array}{ll}
\displaystyle \max(\phi_t + \overline{H}(D\phi), \phi) = 0 & \text{ in } \mathbb{R}^n \times (0, \infty) \\
\displaystyle \phi = \left\{
\begin{array}{ll}
0 & \text{ on } G_0 \times \{0\} \\
-\infty & \text{ on } \mathbb{R}^n \backslash G_0 \times \{0\}.
\end{array} \right.
\end{array} \right.
\end{equation}
$G_0$ is the support of $u_0$, and $\overline{H}(p)$ is an ``effective Hamiltonian'' resulting from the homogenization of \eqref{uequation}. This behavior was shown for a nonlocal equation very close to \eqref{originalequation} that models the propagation of an invasive species in ecology by Perthame and Souganidis in \cite{perthamesouganidis}, and similar asymptotic behavior was found for a non-local Lotka-Volterra equation by Barles, Mirrahimi, and Perthame in \cite{barlesnonlocal}.

Because the behavior of the solutions of \eqref{originalequation} consists of two equilibrium states joined together by a transition layer near the interface defined by \eqref{effectiveequation}, and the effective Hamiltonian $\overline{H}(p)$ can be interpreted as the propagation speed of this interface, our work is connected with the well-studied areas of traveling wave solutions of the KPP equation and the speed of their associated traveling fronts. Recent articles concerning these aspects of nonlocal KPP equations include those by Coville, D\'avila, and Mart\'inez, who in \cite{covilledavilamartinez1} and \cite{covilledavilamartinez2} studied \eqref{originalequation} in the case where $f$ is periodic in $x$. They showed that there exists a critical speed which is the lowest speed for which there exists a pulsating front solution of \eqref{originalequation}. The existence of traveling wave solutions and of a critical speed was considered for a non-local KPP equation similar to \eqref{originalequation} by Berestycki, Nadin, Perthame, and Ryzhik in \cite{berestycki}. Lim and Zlatos in \cite{zlatos2} gave conditions on the inhomogeneity of $f$ in order to prove existence or non-existence of transition fronts for \eqref{originalequation}, where they also studied the range of speeds for which transition fronts exist.

The local version of (\ref{originalequation}), i.e. the equation where the integral term is replaced by a uniformly elliptic second-order operator, has been studied extensively. Its rescaled form reads
\begin{equation} \label{localequation}
u^\epsilon_t - \epsilon a_{ij}\left(x, \frac{x}{\epsilon}\right)u^\epsilon_{ij} + \epsilon^{-1} f\left(x, \frac{x}{\epsilon}, u^\epsilon\right) = 0.
\end{equation}
It was originally studied in the 1930's by Fisher in \cite{fisher} and by Kolmogorov, Petrovskii, and Piskunov in \cite{kpp}. Freidlin in \cite{freidlin} studied the behavior of (\ref{localequation}) using probabilistic methods for the $\frac{x}{\epsilon}$-independent problem. Evans and Souganidis in \cite{souganidisevans} extended \cite{freidlin} and introduced a different approach based on PDE methods which has proven to be more flexible. The asymptotic behavior of $u^\epsilon$ in the presence of periodic space-time oscillation was analyzed by Majda and Souganidis in \cite{majda}. Our work is an extension of \cite{souganidisevans} and \cite{majda} to the nonlocal case. There is also a vast literature dealing with the long-time behavior of \eqref{localequation}, going back to the work of Aronson and Weinberger \cite{aronsonweinberger}.

Due to the presence of the oscillatory variable $\frac{x}{\epsilon}$ in \eqref{uequation}, the theory of homogenization plays a crucial part in the analysis of this equation as $\epsilon \rightarrow 0$. The study of homogenization of Hamilton-Jacobi equations in periodic settings began with the work of Lions, Papanicolaou, and Varadhan \cite{lpv}, and homogenization for ``viscous'' Hamilton-Jacobi equations was studied by Evans \cite{evansptf} and Majda and Souganidis \cite{majda}.  The fundamental tool in the periodic setting is the fact that it is possible to solve the macroscopic problem, or ``cell problem.'' Homogenization in the almost-periodic case was established by Ishii \cite{ishii}, who used the almost periodic structure to construct approximate correctors.

Arisawa in \cite{arisawa} and \cite{arisawa2} studied the periodic homogenization of integro-differential equations with L\'evy operators, equations that are similar in structure to the ones we consider, and we employ the general ideas of her work. She considered the ``ergodic problem'', which is the same as the cell problem, and proved that approximate correctors exist by considering the limit along a subsequence of a family of functions that satisfy an approximated cell problem and showing that the limiting equation satisfies a strong maximum principle. Since such a limiting equation and strong maximum principle are not available in our case, we will use more direct techniques based on an analysis of the nonlocal term to prove the existence of the approximate corrector, and we show that almost-periodicity provides enough of a ``compactness'' criterion in order to make these techniques work.

The paper is organized as follows. In Section \ref{sec:preliminaries}, we make precise our assumptions. In Section \ref{sec:maintheorem}, we state our main result, Theorem \ref{maintheorem}, and we give a heuristic justification for it. In Section \ref{sec:periodic}, we give the proof of homogenization, and in Section \ref{sec:uepsilonconvergence}, we finish the proof of Theorem \ref{maintheorem} using the homogenization result.
\section{Assumptions}
\label{sec:preliminaries}
We assume that $f : \mathbb{R}^{n+1} \rightarrow \mathbb{R}$ is smooth and has bounded derivatives. In particular, it satisfies
\begin{equation} \label{fassumption}
\sup_{x \in \mathbb{R}^n, |u| \leq L} \{|D_{x, u}f(x, u)| + |D^2_{x, u}f(x, u)|\} < \infty \text{ for each } L > 0,
\end{equation}
$D_{x, u}$ denotes derivatives with respect to $x$ and $u$; in the rest of this paper $D$ denotes derivatives with respect to the space variable $x$. We also assume that $f$ is of KPP type i.e. monostable in the $u$ variable. That is, for every $x \in \mathbb{R}^n$, $f$ satisfies
\begin{equation} \label{monostable}
\left\{
\begin{array}{rl}
f(x, u) < 0 & \text{for } u \in (-\infty, 0) \cup (1, \infty),\\
f(x, u) > 0 & \text{for } u \in (0, 1), \end{array}\right.
\end{equation}
and
\begin{equation} \label{kppcondition}
c(x) := \frac{\partial f}{\partial u}(x, 0)= \sup_{u > 0} u^{-1} f(x, u) \geq \kappa > 0.
\end{equation}
Because (\ref{fassumption}) implies that $f(x, u)$ is smooth and has locally in $u$ and globally in $x$ bounded first and second derivatives in both variables, we can see that $c(x)$ is smooth, bounded, and Lipschitz continuous. Define $K := \max(\|c(x)\|_\infty, \|Dc\|_\infty)$.

Concerning the kernel $J$, we assume that
\begin{equation} \label{Jassumption}
\left\{
\begin{array}{l}
 J \text{ is compactly supported in a set } O \subset B(0, \bar{r}), J \geq 0, \\
 J \in C(\mathbb{R}^n), J(x) = J(-x) \text{ for all } x \in \mathbb{R}^n, \int_{\mathbb{R}^n} J(y) dy = \bar{J} < \infty,\\
 \text{There exists } r_1 > 0 \text{ such that } J(y) \geq A > 0 \text{ on } B(0, r_1).\end{array} \right.
\end{equation}
Concerning the initial condition $u_0$, we assume that
\begin{equation} \label{u0assumption}
u_0 \in C(\mathbb{R}^n), 0 \leq u_0 \leq 1, \text{ and } G_0 := \mathrm{spt}(u_0) = \overline{\{x \ |\ u_0(x) \neq 0\}} \text{ is compact}.
\end{equation}
We assume that the nonlinearity is almost-periodic, that is, we assume that the family
\begin{equation} \label{almostperiodic}
\{c(\cdot + z) : z \in \mathbb{R}^n\} \text{ is relatively compact in } \mathrm{BUC}(\mathbb{R}^n).
\end{equation}
Note that the typical assumption of 1-periodicity is a specific case of almost-periodicity.

\section{Main Result, Heuristic Derivation}
\label{sec:maintheorem}
We now state our main result.
\begin{theorem} \label{maintheorem}
Assume \eqref{fassumption}-\eqref{almostperiodic}. Then there exists a continuous function $\overline{H}: \mathbb{R}^n \rightarrow \mathbb{R}$ such that as $\epsilon \rightarrow 0$, $u^\epsilon \rightarrow 0$ in $\{\phi < 0\}$ and $u^\epsilon \rightarrow 1$ in $\mathrm{int}\{\phi = 0\}$ locally uniformly, where $\phi$ is the unique solution of \eqref{effectiveequation}.
\end{theorem}
Next we explain in a heuristic way the origin of the variational inequality and why it controls the asymptotic behavior of the $u^\epsilon$. Following the work for local KPP equations mentioned in the introduction, we now use the classical Hopf-Cole transformation
\begin{equation} \label{hopfcole}
u^\epsilon = \exp(\epsilon^{-1} \phi^\epsilon).
\end{equation}
It is immediate that for $t = 0$, $\phi^\epsilon = -\infty$ on $\mathbb{R}^n \backslash G_0$ and $\phi^\epsilon \rightarrow 0$ on $G_0$ as $\epsilon \rightarrow 0.$ The interesting part of the transformation comes into play for $t > 0$. We can see via straightforward calculations that $\phi^\epsilon$ solves
$$
\phi^\epsilon_t(x, t) + \bar{J} - \int J(y) \exp\left(\frac{\phi^\epsilon(x - \epsilon y, t) - \phi^\epsilon(x, t)}{\epsilon}\right) dy - \frac{1}{u^\epsilon} f\left(\frac{x}{\epsilon}, u^\epsilon\right) = 0,
$$
an equation which can be analyzed using homogenization techniques. We assume that $\phi^\epsilon$ admits the asymptotic expansion $\phi^\epsilon(x, t) = \phi(x, t) + \epsilon v(\frac{x}{\epsilon}) + O(\epsilon^2).$ Writing $z = \frac{x}{\epsilon}$ and performing a formal computation, we obtain
$$
\phi_t + \bar{J} - \int J(y) \exp\left(\frac{\phi(x - \epsilon y, t) -\phi(x, t)}{\epsilon} + v(z-y) - v(z))\right) dy - \frac{1}{u^\epsilon} f\left(z, u^\epsilon\right)  = 0.
$$
Formally, we can say that as $\epsilon \rightarrow 0$, $\epsilon^{-1}(\phi(x - \epsilon y, t) -\phi(x, t)) \rightarrow -y\cdot D\phi(x, t).$ In addition, if $u^\epsilon \rightarrow 0$ as $\epsilon \rightarrow 0$, then
\begin{equation}
(u^\epsilon)^{-1} f(z, u^\epsilon) \rightarrow \displaystyle \frac{\partial f}{\partial u}(z, 0) = c(z). \label{blah2}
\end{equation}
Writing $p = D\phi(x, t)$, we see that oscillatory behavior disappears in the limit as $\epsilon \rightarrow 0$ if it is possible to find a constant $\overline{H}(p)$ and a function $v$ that solves
\begin{equation} \label{cellproblem}
\bar{J}-\int J(y) \exp(-y\cdot p) \exp(v(z-y) - v(z)) dy - c(z) = \overline{H}(p),
\end{equation}
which is a typical macroscopic problem or ``cell problem'' from homogenization theory. The issue is to find $\overline{H}(p)$, referred to as the effective Hamiltonian, so that (\ref{cellproblem}) admits a solution $v$, typically referred to as a ``corrector,'' with appropriate behavior at infinity i.e. strict sublinearity, so that $\overline{H}(p)$ is unique. If an effective Hamiltonian and a corresponding corrector can be found, then we see that $\phi^\epsilon$ converges to a function $\phi$ that satisfies $\phi_t + \overline{H}(D\phi) = 0,$ provided that we also ensure that $\phi < 0$ so $u^\epsilon \rightarrow 0$ due to \eqref{hopfcole}, which would then allow us to apply \eqref{blah2}.
Therefore, $\phi$ should satisfy the Hamilton-Jacobi variational inequality
$$
\max(\phi_t + \overline{H}(D\phi), \phi) = 0,
$$
which combined with the initial condition at $t = 0$ is precisely \eqref{effectiveequation}. Then \eqref{hopfcole}, the fact that $\phi^\epsilon \rightarrow \phi$, and an additional argument, found in Section \ref{sec:uepsilonconvergence}, to show that $u^\epsilon \rightarrow 1$ on the set $\{\phi = 0\}$ imply that $u^\epsilon$ satisfies the behavior described by Theorem \ref{maintheorem}.
\section{Proof of Homogenization}
\label{sec:periodic}
We proceed to prove Theorem \ref{maintheorem} rigorously. Our primary result in this section is the homogenization of \eqref{phiequation}, that is, we show that solutions $\phi^\epsilon$ of
\begin{equation} \label{phiequation}
\left\{
\begin{array}{ll}
\displaystyle \phi^\epsilon_t + \bar{J} - \int J(y) \exp\left(\frac{\phi^\epsilon(x - \epsilon y, t) - \phi^\epsilon(x, t)}{\epsilon}\right) dy - \frac{f\left(\frac{x}{\epsilon}, \exp(\epsilon^{-1} \phi^\epsilon)\right)}{\exp(\epsilon^{-1}\phi^\epsilon)} = 0 & \text{ on } \mathbb{R}^n \times (0, \infty) \\
\displaystyle \phi^\epsilon = \epsilon \log(u_0) & \text{ on } \mathrm{int}(G_0) \times \{0\} \\
\displaystyle \phi^\epsilon(x, t) \rightarrow -\infty \text{ as } t \rightarrow 0 & \text{ for } x \in \mathbb{R}^n \backslash G_0
\end{array} \right.
\end{equation}
converge locally uniformly to $\phi$, the solution of the homogenized equation \eqref{effectiveequation}.
\begin{theorem} \label{periodichomogtheorem}
Under the assumptions \eqref{fassumption}-\eqref{almostperiodic}, $\phi^\epsilon$ converges locally uniformly to $\phi$ as $\epsilon \rightarrow 0$ on $\mathbb{R}^n \times (0, \infty)$.
\end{theorem}

Our first objective is to find $\overline{H}(p)$ such that the cell problem (\ref{cellproblem}) admits ``approximate correctors'' $v^+, v^-$ that satisfy (\ref{approxcorrector1}) and (\ref{approxcorrector2}) respectively, as the existence of approximate correctors is sufficient to prove homogenization. The proofs in the almost-periodic and periodic cases are very similar, so we will present the proof in the almost periodic case and explain how the proof differs in the periodic case.

We start by making the typical approximation to the cell problem (see \cite{lpv}) and consider the following equation in $\mathbb{R}^n$ for $\lambda > 0$:
\begin{equation} \label{approxcellproblem}
\lambda v^\lambda(z) + \bar{J}-\int J(y) \exp(-y\cdot p) \exp(v^\lambda(z-y) - v^\lambda(z)) dy - c(z) = 0.
\end{equation}
First we need to show that this problem is well-posed. The proof follows along similar lines of other comparison proofs (see \cite{nonlocalcomparison}, \cite{arisawacomparison}, \cite{nonlocaldirichlet}, \cite{barlesimbert}, \cite{usersguide}, \cite{ishii1994viscosity}).
\begin{proposition}
Let $u(z) \in \mathrm{USC}(\mathbb{R}^n)$ be a bounded subsolution of \eqref{approxcellproblem}, and let $v(z) \in \mathrm{LSC}(\mathbb{R}^n)$ be a bounded supersolution of \eqref{approxcellproblem}. Then $u \leq v$ in $\mathbb{R}^n$. In addition, there exists a unique bounded continuous solution of \eqref{approxcellproblem}.
\end{proposition}
\begin{proof}
We first prove that comparison holds. Assume for a contradiction that $M:=\sup_{z \in \mathbb{R}^n} u(z) - v(z) > 0.$ Then define
$$
M_\delta := \max_{z \in \mathbb{R}^n} u(z) - v(z) - \delta |z|^2.
$$
Note that this quantity is positive for $\delta$ sufficiently small. Because of our assumption that $u, v$ are bounded, there exists a point $z_\delta$ such that $M_\delta = u(z_\delta) - v(z_\delta) - \delta |z_\delta|^2$.
We can deduce that
\begin{equation} \label{comparisonthing}
\lim_{\delta \rightarrow 0} M_\delta = M, \lim_{\delta \rightarrow 0} \delta |z_\delta|^2 = 0.
\end{equation}
Because $u, v$ are respectively a subsolution and a supersolution of \eqref{approxcellproblem}, we obtain
\begin{align*}
\lambda u(z_\delta) &+ \bar{J}-\int J(y) \exp(-y\cdot p) \exp(u(z_\delta-y) - u(z_\delta)) dy - c(z_\delta) \leq 0 \\
\lambda v(z_\delta) &+ \bar{J}-\int J(y) \exp(-y\cdot p) \exp(v(z_\delta-y) - v(z_\delta)) dy - c(z_\delta) \geq 0.
\end{align*}
Subtracting the second inequality from the first, we have
\begin{multline} \label{comparisonthing2}
\lambda(u(z_\delta) - v(z_\delta)) - \int J(y) \exp(-y\cdot p) \\ [\exp(u(z_\delta-y) - u(z_\delta)) - \exp(v(z_\delta-y) - v(z_\delta))] dy \leq 0.
\end{multline}
We know that for any $y \in \mathbb{R}^n$,
\begin{equation} \label{comparisonthing3}
u(z_\delta - y) - u(z_\delta) \leq v(z_\delta - y) - v(z_\delta) - \delta[2 z_\delta \cdot y - |y|^2].
\end{equation}
\eqref{comparisonthing} implies that as $\delta \rightarrow 0$, for any $y \in \mathbb{R}^n$, $\delta[2 z_\delta \cdot y - |y|^2] \rightarrow 0$. Therefore, because $u, v$ are bounded, and $y$ is contained in a ball $B(0, \bar{r})$ in the integral term of \eqref{comparisonthing2}, we can apply \eqref{comparisonthing3} to \eqref{comparisonthing2} and take the limit $\delta \rightarrow 0$ to get $\lambda(u(z_\delta) - v(z_\delta)) \leq o_\delta(1),$ but this is a contradiction because the left hand side is uniformly positive by \eqref{comparisonthing}. Therefore, $M \leq 0$, which was what we wanted to show, and this completes the proof of comparison/uniqueness. The existence of a bounded continuous solution given a comparison principle is a consequence of Perron's method, and follows in the same way as the analogous result in \cite{nonlocalcomparison}.
\end{proof}
The next proposition, which shows that there exist approximate correctors to the cell problem, is the main objective of this section. It is similar to the analogous one found in \cite{arisawa}. In that work Arisawa considers the ``ergodic problem'', which is essentially the statement of Proposition \ref{periodictheorem}, for a different nonlocal equation, a periodic integro-differential equation containing a L\`evy operator, that bears resemblance to \eqref{approxcellproblem}. In our proof, we will employ some techniques from \cite{arisawa} along with some new ones involving an analysis of the nonlocal term of \eqref{approxcellproblem}.

For the almost periodic setting, we introduce the concept of ``uniform almost periodicity''.
\begin{defn} \label{unifalmostperiodic}
The collection of functions $\{f_s\}_{s \in I}$ for $I$ an arbitrary index set is uniformly almost periodic in $s$ if given any sequence $\{x_j\} \in \mathbb{R}^n$ there exists a subsequence, also called $\{x_j\}$ for convenience, such that for any $\epsilon > 0$ there exists $N$ such that for all $s \in I$ and all $j, k \geq N$,
$$
\| f_s(x_j + \cdot) - f_s(x_k + \cdot) \|_\infty < \epsilon.
$$
\end{defn}
This definition means that for $\{f_s\}$ the almost periodicity condition \eqref{almostperiodic} holds uniformly in $s$. This concept will be used to give sufficient ``compactness'' for the almost periodic setting in order to apply the techniques that are applicable to the periodic setting.
\begin{proposition} \label{periodictheorem}
Assume \eqref{fassumption}-\eqref{Jassumption} and \eqref{almostperiodic}. For any $p \in \mathbb{R}^n$, there exists a unique $\overline{H}(p)$ such that for each $\nu > 0$ there exist bounded, Lipschitz continuous functions $v^+, v^-$ solving
\begin{align}
\bar{J} - \int J(y) \exp(-y \cdot p) \exp (v^+(z-y) - v^+(z)) dy - c(z) &\leq \overline{H}(p) + \nu \label{approxcorrector1}\\
\bar{J} - \int J(y) \exp(-y \cdot p) \exp (v^-(z-y) - v^-(z)) dy - c(z) &\geq \overline{H}(p) - \nu, \label{approxcorrector2}.
\end{align}
In addition,
\begin{equation} \label{periodicconvergence}
\displaystyle \lim_{\lambda \downarrow 0} \lambda v_{\lambda}(z) = -\overline{H}(p)
\end{equation}
uniformly in $\mathbb{R}^n$, where $v^\lambda$ is the solution to \eqref{approxcellproblem}.
\end{proposition}
\begin{proof}
We first show that if there exists a constant $\overline{H}(p)$ such that for every $\nu > 0$, there exist functions $v^+, v^-$ satisfying Proposition \ref{periodictheorem}, then $\overline{H}(p)$ is unique. This argument was originally found in \cite{lpv}.
\begin{lemma} \label{uniquenesslemma}
Suppose that there exists $\overline{H}(p)$ such that for any $\nu > 0$ there exist bounded, Lipschitz functions $v^+, v^-$ satisfying \eqref{approxcorrector1} and \eqref{approxcorrector2} respectively. Then $\overline{H}(p)$ is unique.
\end{lemma}
\begin{proof}
Suppose for a contradiction that there exists $A < B$ such that for any $\nu > 0$ there exist bounded $v^{+}_\nu, v^{-}_\nu$ that satisfy the following for all $z \in \mathbb{R}^n$:
\begin{align*}
\bar{J} -\int J(y) \exp(-y \cdot p) \exp (v^+_\nu(z-y) - v^+_\nu(z)) dy - c(z) &\leq  A + \nu \\
\bar{J} -\int J(y) \exp(-y \cdot p) \exp (v^-_\nu(z-y) - v^-_\nu(z)) dy - c(z) &\geq  B - \nu.
\end{align*}
Fix a sufficiently small $\nu$ such that $B - 2\nu > A + 2\nu$. Because $v^+_\nu$ and $v^-_\nu$ are bounded, we can add a constant to $v^+_\nu$ to ensure that $v^+_\nu > v^-_\nu$ on $\mathbb{R}^n$. In addition, for $\epsilon$ sufficiently small,
\begin{multline} \label{blah4}
\epsilon v^-_\nu + \bar{J} -\int J(y) \exp(-y \cdot p) \exp (v^-_\nu(z-y) - v^-_\nu(z)) dy - c(z) \geq B - 2\nu \\
> A + 2\nu \geq \epsilon v^+_\nu + \bar{J} -\int J(y) \exp(-y \cdot p) \exp (v^+_\nu(z-y) - v^+_\nu(z)) dy - c(z)
\end{multline}
holds for all $z \in \mathbb{R}^n$.  Now comparison, which can be applied for \eqref{blah4} because $v^+_\nu$ and $v^-_\nu$ are bounded and Lipschitz continuous, now implies that $v^-_\nu \geq v^+_\nu$. This is a contradiction, and thus $B = A$ and so $\overline{H}$ is unique.
\end{proof}
Now we proceed with proving that there exists such a constant $\overline{H}(p)$. Let $z_0 \in \mathbb{R}^n$ be fixed, and define $w^\lambda(z) := v^\lambda(z) - v^\lambda(z_0) \text{ and } C_\lambda := \lambda \|w^\lambda\|_\infty.$ We know that $C_\lambda$ is bounded due to comparison for (\ref{approxcellproblem}) between $\lambda w^\lambda$ and constant functions depending on $\sup_{\mathbb{R}^n} |c(\cdot)|$.
We claim that if $C_{\lambda} \rightarrow 0$ as $\lambda \rightarrow 0$, then the proposition follows. This is true because $\|\lambda v_{\lambda}(z) - \lambda v_{\lambda}(z_0)\|_\infty = \lambda \|w_{\lambda}\|_\infty \rightarrow 0.$ Because $\lambda v^\lambda(z_0)$ is bounded in $\lambda$, there exists a subsequence such that we can define $\overline{H}(p) := \lim_{\lambda \rightarrow 0} -\lambda v^\lambda(z_0),$ such that \eqref{periodicconvergence} holds. Now we can see that upon taking $\lambda$ sufficiently small so that $\|v^\lambda - \overline{H}(p)\|_\infty < \nu$, $v^\lambda$ satisfies (\ref{approxcorrector1}) and (\ref{approxcorrector2}). Then Lemma \ref{uniquenesslemma} allows us to finish the proof of Proposition \ref{periodictheorem} in this case.

Therefore, suppose for a contradiction that $\displaystyle \liminf_{\lambda \rightarrow 0} C_{\lambda} > 0$. Since $C_\lambda$ is uniformly bounded, we can extract a sequence $\lambda_n \rightarrow 0$ such that $\lim_{n \rightarrow \infty} C_{\lambda_n} = C' > 0.$ We will subsequently call this subsequence $\lambda$ for convenience. Now define
$$\tilde{w}^\lambda(z) = \frac{w^\lambda(z)}{\|w^\lambda\|_\infty}.$$
Then we have that upon writing $\tilde{c}(z) = c(z) - \lambda v^\lambda (z_0)$, $\tilde{w}^\lambda$ satisfies
\begin{equation}
\label{wequation2}
\lambda \tilde{w}^\lambda + \frac{\lambda}{C_\lambda}\bar{J} - \frac{\lambda}{C_\lambda} \int J(y) \exp(-y\cdot p) \exp\left(\frac{C_\lambda}{\lambda}(\tilde{w}^\lambda(z-y) - \tilde{w}^\lambda(z))\right) dy - \frac{\lambda}{C_\lambda}\tilde{c}(z) = 0.
\end{equation}
Our objective is to show that $\tilde{w}^\lambda$ converges uniformly to zero. Assume for a contradiction that there exist sequences $\lambda_j, z_j$ such that $\lambda_j \rightarrow 0 \text{ and } \tilde{w}^{\lambda_j}(z_j) \rightarrow \delta \neq 0,$ and suppose without loss of generality that $\delta$ is positive. We claim that there exists a point $\hat{z}$ such that for all $j$ sufficiently large,
\begin{equation} \label{zhatcondition}
\tilde{w}^{\lambda_j}(\hat{z}) \geq \frac{\delta}{4}.
\end{equation}
In the case where $c(z)$ is periodic, $\tilde{w}^\lambda$ is also periodic, and then a point $\hat{z}$ satisfying \eqref{zhatcondition} can be found by compactness because the sequence $\{z_j\}$ can be taken to lie in the unit cube. In the case where $c(z)$ is almost-periodic, we use the fact that the family $\{\tilde{w}^\lambda\}_{\lambda \leq 1}$ is uniformly almost periodic in $\lambda$ in the sense of Definition \ref{unifalmostperiodic}. This is true because by separated $z$ dependence and comparison for (\ref{wequation2}), we know that for any $x_1, x_2 \in \mathbb{R}^n$ and any $\lambda \leq 1$, there exists a uniform constant $C$ such that
$$
|\tilde{w}^\lambda(\cdot+x_1) - \tilde{w}^\lambda(\cdot + x_2)| \leq C|\tilde{c}(\cdot + x_1) - \tilde{c}(\cdot + x_2)|,
$$
and now because $\tilde{c}(z)$ is uniformly almost periodic in $\lambda$, which follows from the fact that $c(z)$ is an almost periodic function, we have that the family $\{\tilde{w}^\lambda\}_{\lambda \leq 1}$ is uniformly almost periodic.

Because $\{\tilde{w}^\lambda\}_{\lambda \leq 1}$ is uniformly almost periodic, we can extract a subsequence of $\{z_j\}$, also called $\{z_j\}$, and take $N$ sufficiently large so that
\begin{equation} \label{wlalmostperiodic}
|\tilde{w}^\lambda(z_j + z) - \tilde{w}^\lambda(z_k + z)| \leq \frac{\delta}{2}
\end{equation}
for all $j, k \geq N$, all $z \in \mathbb{R}^n$, and for all $\lambda \leq 1$. Now if we fix $k \geq N$, then for $j$ sufficiently large, the fact that $\tilde{w}^\lambda(z_j) \rightarrow \delta$ and (\ref{wlalmostperiodic}) applied with $z = 0$ implies $\tilde{w}^{\lambda_j}(z_k) \geq \frac{\delta}{4}$, so $z_k$ is a point $\hat{z}$ that satisfies \eqref{zhatcondition}.

Now we use \eqref{zhatcondition} and $\tilde{w}^\lambda(z_0) = 0$ to reach a contradiction. Note that since the integrand of the nonlocal term is always nonnegative, we can restrict our integration domain to suitable subsets when seeking lower bounds. We consider (\ref{wequation2}) as $\lambda_j \rightarrow 0$.  We have that $\lambda \tilde{w}^\lambda \rightarrow 0$ uniformly because $\|\tilde{w}^\lambda\|_\infty = 1$, and there exists a constant $C < \infty$ such that
$$\left\|\frac{\lambda \tilde{c}(z)}{C_\lambda}\right\|_\infty = \left\|\frac{1}{C_\lambda}(\lambda c(z) - \lambda^2 v^\lambda(0))\right\|_\infty \leq C$$
because $C_\lambda \rightarrow C' > 0$ as $\lambda \rightarrow 0$. Therefore we consider the nonlocal second term
\begin{equation*}
W^\lambda(z) := -\frac{\lambda}{C_\lambda} \int J(y) \exp(-y\cdot p) \exp\left(\frac{C_\lambda}{\lambda}(\tilde{w}^\lambda(z-y) - \tilde{w}^\lambda(z))\right) dy.
\end{equation*}
If we consider the line between $\hat{z}$ and $z_0$ and cover it with $M:= \frac{3|\hat{z} - z_0|}{r_1}$ balls of radius $\frac{r_1}{3}$, then because $\tilde{w}^{\lambda_j}$ increases by at least $\frac{\delta}{4}$ on that line from $z_0$ to $\hat{z}$, then there exists $x_j \in \mathbb{R}^n$ such that
$$
\osc_{B(x_{j}, \frac{r_1}{3})} \tilde{w}^{\lambda_j} \geq \delta_2 := \frac{\delta}{4M},
$$
Write $A_{j} = B(x_j, \frac{r_1}{3})$, and consider $x_{j, min}, x_{j, max} \in \overline{A}_j$ to be the points respectively where $\tilde{w}^{\lambda_j}$ is minimized and maximized over $\overline{A}_j$. Then we know that $\tilde{w}^{\lambda_j}(x_{j, max}) - \tilde{w}^{\lambda_j}(x_{j, min}) \geq \delta_2.$ In addition, we have due to comparison for (\ref{wequation2}) and the separated $z$ dependence, $\tilde{w}^\lambda$ is Lipschitz continuous with constant $K_2 = \frac{2K}{C'}$ for all $\lambda$ sufficiently small. This gives us
\begin{equation} \label{blah7}
\tilde{w}^{\lambda_j}(x_{j, max} - y) - \tilde{w}^{\lambda_j}(x_{j, min}) \geq \frac{\delta_2}{2}\text{ for } y \in B\left(0, r_2\right), r_2 := \frac{\delta_2}{2K_2}.
\end{equation}
Finally, to obtain a contradiction, we consider $W^{\lambda_j}(z)$ with $z = x_{j, min}$, and we define $A_2 := B(x_{j, min} - x_{j, max}, \min(r_2, r_1 - |x_{j, min} - x_{j, max}|)).$ $A_2$ is contained in $B(0, r_1)$ by construction, and its radius is positive because $|x_{j, min} - x_{j, max}| < \frac{2r_1}{3}$.  In addition, for $y \in A_2$, (\ref{blah7}) implies that $\tilde{w}^{\lambda_j}(x_{j, min} -y) - \tilde{w}^{\lambda_j}(x_{j, min}) \geq \frac{\delta_2}{2}.$ This gives us
\begin{align*}
W^{\lambda_j}(x_{j, min}) &\geq \frac{\lambda_j}{C_{\lambda_j}} \int_{A_2} J(y)\exp(-y\cdot p) \exp\left(\frac{C_{\lambda_j}}{\lambda_j}(\tilde{w}^{\lambda_j}(x_{j, min}-y) - \tilde{w}^{\lambda_j}(x_{j, min}))\right) dy \\
&\geq C_1 \lambda_j \int_{A_2} J(y) \exp(-y \cdot p) \exp\left(\frac{C_{\lambda_j} \delta_2}{2\lambda_j}\right) dy \geq C_1\lambda_j \exp\left(\frac{C_2}{\lambda_j}\right),
\end{align*}
where $C_1, C_2 > 0$ are constants. $C_1\lambda_j \exp(\frac{C_2}{\lambda_j})$ is unbounded as $\lambda_j \rightarrow 0$, which means that $W^\lambda$ is unbounded, a contradiction to (\ref{wequation2}). Therefore, $\tilde{w}^\lambda$ converges uniformly to zero as $\lambda \rightarrow 0$, but $\|\tilde{w}^\lambda\|_\infty = 1$ for all $\lambda$ by construction, a contradiction. Therefore, $C_\lambda \rightarrow 0$, and we have the existence of $\overline{H}$ satisfying \eqref{periodicconvergence}. In addition, given $\nu > 0$, we also have the existence of bounded, Lipschitz continuous $v^+$ and $v^-$ satisfying (\ref{approxcorrector1}) and (\ref{approxcorrector2}) respectively, because we can simply take $v^+ = v^- = v^\lambda$ for $\lambda$ sufficiently small depending on $\nu$. This concludes the proof of Proposition \ref{periodictheorem} because this means that any convergent subsequence (in the uniform metric) of $\lambda v^\lambda(\cdot)$ must converge to $-\overline{H}(p)$, and so the full sequence $\lambda v^\lambda(x)$ converges uniformly to $-\overline{H}(p)$, which is unique by Lemma \ref{uniquenesslemma}.
\end{proof}
\begin{remark}
Arisawa in \cite{arisawa} proves that the analogue to $\tilde{w}^\lambda$ in her (periodic) setting converges uniformly to zero by using the uniform equicontinuity of $\tilde{w}^\lambda$ to find a limiting function $\tilde{w}$ along a subsequence as $\lambda \rightarrow 0$ via Arzela-Ascoli. She subsequently shows that $\tilde{w}^\lambda \rightarrow 0$ by proving that $\tilde{w}$ solves an equation that has a strong maximum principle. The techniques demonstrated in the above proof overcome the fact that in our case taking $\lambda \rightarrow 0$ in \eqref{wequation2} does not yield such an equation.
\end{remark}
We now move to the proof of Theorem \ref{periodichomogtheorem}. We first prove a technical lemma which supplies bounds on $u^\epsilon$, the solution of \eqref{uequation}, and $\phi^\epsilon$, the solution of \eqref{phiequation}. Note that because $\phi^\epsilon$ is given by \eqref{hopfcole}, since we know that \eqref{uequation} is well posed (see \cite{nonlocalcomparison}), and in particular that a comparison principle holds, we know that a comparison principle holds for \eqref{phiequation} as well.
\begin{lemma} \label{kpplemma}
Assume that $f$ satisfies (\ref{fassumption})-(\ref{kppcondition}), $J$ satisfies (\ref{Jassumption}), $u_0$ satisfies (\ref{u0assumption}), and $\epsilon < 1$. Then
\begin{compactenum}
    \item $0 \leq u^\epsilon \leq 1$ on $\mathbb{R}^n \times [0, \infty)$.
    \item For each compact subset $Q$ of $[\mathbb{R}^n \times (0, \infty)] \cup [\mathrm{int}(G_0) \times \{0\}]$, there exists a constant $C(Q)$, independent of $\epsilon$, such that
        \begin{equation} \label{localphiepsilonbound}
        |\phi^\epsilon| \leq C(Q) \text{ on } Q.
        \end{equation}
\end{compactenum}
\end{lemma}
\begin{proof}
The first part is a consequence of comparison for (\ref{uequation}) and the fact that the constant functions 0 and 1 are respectively a subsolution and a supersolution of (\ref{uequation}). To prove the second part, note that it suffices to show a lower bound because $u^\epsilon \leq 1$ implies that $\phi^\epsilon \leq 0$. To do this, we adapt the argument from Lemma 2.1 of \cite{souganidisevans}. First, we can assume without loss of generality that there exists an $R > 0$ such that $B(0, R) \subset \mathrm{int}(G_0) \text{ and } \inf_{B(0, R)} u_0 > 0.$ We first show that $\phi^\epsilon$ is bounded from below on $B(0, R) \times (0, \infty)$. To this end, define the function $\varphi_1: \mathbb{R}^n \times [0, \infty) \rightarrow \mathbb{R} \cup \{-\infty\}$ by
$$
\varphi_1(x, t) := \begin{cases}
\frac{1}{|x|^2 - R^2} - \alpha t - \beta \text{ if } (x, t) \in B(0, R) \times [0, \infty) \\
-\infty \text{ if } (x, t) \in \mathbb{R}^n \backslash B(0, R) \times [0, \infty),
\end{cases}
$$
where $\alpha, \beta$ are positive constants to be chosen. We can now compute
\begin{align*}
\displaystyle \varphi_{1, t} + \bar{J} &- \int J(y) \exp\left(\frac{\varphi_1(x - \epsilon y, t) -
\varphi_1(x, t)}{\epsilon}\right) dy - \frac{f\left(\frac{x}{\epsilon}, \exp(\epsilon^{-1}
\varphi_1)\right)}{\exp(\epsilon^{-1}\varphi_1)} \\
&\leq -\alpha + \bar{J} - \int J(y) \exp\left(\frac{\varphi_1(x - \epsilon y, t) -
\varphi_1(x, t)}{\epsilon}\right) dy - \kappa \leq -\alpha + \bar{J} - \kappa,
\end{align*}
where the second inequality follows due to \eqref{kppcondition}. Therefore, upon taking $\alpha$ sufficiently large, we can make the right hand side negative on $B(0, R) \times (0, \infty)$. If we take $\beta = \log(\inf_{B(0, R)} u_0),$ then we have that $\varphi_1 \leq \phi^\epsilon \text{ on } B(0, R)^c \times (0, \infty) \cup B(0, R) \times \{0\}.$ Now comparison for \eqref{phiequation} implies that $\varphi_1 \leq \phi^\epsilon$ on $B(0, R) \times (0, \infty)$, which means that
\begin{equation} \label{partialbound}
|\phi^\epsilon| \leq C \text{ on } \overline{B}\left(0, \frac{R}{2}\right) \times (0, \infty)
\end{equation}
We now provide a lower bound for the points $(x, t) \in \mathbb{R}^n \times (0, \infty)$ such that $|x| > \frac{R}{2}$. Define
$$
\varphi_2(x, t) := \begin{cases}
-\frac{\gamma|x|^2}{t} - \sigma t - \tau \text{ if } t > 0 \\
-\infty \text{ if } t = 0,
\end{cases}
$$
where $\gamma, \sigma, \tau$ are positive constants to be determined. We can compute
\begin{align*}
\displaystyle \varphi_{2, t} + \bar{J} &- \int J(y) \exp\left(\frac{\varphi_2(x - \epsilon y, t) -
\varphi_2(x, t)}{\epsilon}\right) dy - \frac{f\left(\frac{x}{\epsilon}, \exp(\epsilon^{-1}
\varphi_2)\right)}{\exp(\epsilon^{-1}\varphi_2)} \\
&\leq -\sigma + \frac{\gamma |x|^2}{t^2}  - \int J(y) \exp\left(\frac{-\gamma(-2(x\cdot y) + \epsilon |y|^2)}{t}\right) dy - \kappa \\
&= -\sigma - \left(\int J(y) \exp\left(\frac{2\gamma (x \cdot y) - \gamma \epsilon |y|^2)}{t}\right) dy - \frac{\gamma |x|^2}{t^2} \right) - \kappa \\
&= -\sigma - \left(\int J(y) \exp\left(\frac{\gamma ((x- \frac{\epsilon y}{2}) \cdot y)}{t}\right)^2 dy - \left(\frac{\sqrt{\gamma} |x|}{t}\right)^2\right) - \kappa \\
&\leq 0 \text{ on } [\mathbb{R}^n \backslash \overline{B}(0, \frac{R}{2}) \times (0, \infty)].
\end{align*}
We justify the last inequality. It suffices to show that
\begin{equation} \label{blah12}
\int J(y) \exp\left(\frac{\gamma ((x- \frac{\epsilon y}{2}) \cdot y)}{t}\right) dy - \left(\frac{\sqrt{\gamma} |x|}{t}\right)\leq 0.
\end{equation}
By \eqref{Jassumption} there exists $r_2 > 0$ such that
$$
\int_{B(0, \frac{R_2}{2}) \backslash B(0, r_2)} J(y) = C_2 > 0,
$$
so by the symmetry of $J$, we know that there is a positive constant $C_3$ such that
$$
\int_{B(0, \frac{R_2}{2}) \backslash B(0, r_2)} \frac{J(y)((x- \frac{\epsilon y}{2}) \cdot y)}{t} dy \geq \frac{C_3|x|}{t}.
$$
Therefore, if we take $\gamma$ sufficiently large, we have that \eqref{blah12} is satisfied, and thus $\varphi_2$ is a subsolution of \eqref{phiequation}. Select $\tau$ to be larger than the constant from (\ref{partialbound}), and define $\hat{\phi}^\epsilon(x, t) := \phi^\epsilon(x, t + \xi)$ for $\xi > 0$. Then we have that
$$\varphi_2 \leq \hat{\phi}^\epsilon \text{ on }
\left[\overline{B}\left(0, \frac{R}{2}\right) \times [0, \infty)\right] \cup \left[\overline{B}\left(0, \frac{R}{2}\right)^c \times \{0\}\right].$$
This means that we can apply comparison to conclude that $\varphi_2 \leq \hat{\phi}^\epsilon \text{ on } \overline{B}(0, \frac{R}{2})^c \times (0, \infty)$,
and taking $\xi \rightarrow 0$ gives us \eqref{localphiepsilonbound}.
\end{proof}
We will now use the perturbed test function method (see \cite{evansptf}, \cite{arisawa2}) to prove Theorem \ref{periodichomogtheorem}.
\begin{proof}[Proof of Theorem \ref{periodichomogtheorem}]
For each $(x, t) \in \mathbb{R}^n \times (0, \infty)$, we define
\begin{equation} \label{phistar}
\phi^*(x, t) = \limsup_{\epsilon \rightarrow 0, (x', s) \rightarrow (x, t)} \phi^\epsilon(x', s), \phi_*(x, t) = \liminf_{\epsilon \rightarrow 0, (x', s) \rightarrow (x, t)} \phi^\epsilon(x', s)
\end{equation}
to be the half-relaxed upper and lower limits (see \cite{usersguide}); note that the local uniform bounds on $\phi^\epsilon$ from Lemma \ref{kpplemma} implies that $\phi^*(x, t), \phi_*(x, t) \in \mathbb{R}$ for all $(x, t) \in \mathbb{R}^n \times (0, \infty)$. We first show that $\phi^*$ is a solution of
\begin{equation} \label{variationalinequality}
\max(\phi^*_t + \overline{H}(D\phi^*), \phi^*) \leq 0 \text{ in } \mathbb{R}^n \times (0, \infty).
\end{equation}
Because $\phi^\epsilon \leq 0$ by (\ref{hopfcole}) and Lemma \ref{kpplemma}, showing (\ref{variationalinequality}) reduces to showing
\begin{equation} \label{homogenizedHJequation}
\phi^*_t + \overline{H}(D\phi^*) \leq 0 \text{ in } \mathbb{R}^n \times (0, \infty)
\end{equation}
in the viscosity sense. Take a smooth test function $\varphi$ and a point $(x_0, t_0) \in \mathbb{R}^n \times (0, \infty)$ such that $(x, t) \mapsto \phi^*(x, t) - \varphi(x, t)$ has a strict global maximum at $(x_0, t_0)$ (see \cite{usersguide}). Assume for a contradiction that $\varphi_t(x_0, t_0) + \overline{H}(D\varphi(x_0, t_0)) = \theta > 0.$ Set $p_0 := D\varphi(x_0, t_0)$, and define the perturbed test function
\begin{equation} \label{perturbedtestfunction}
\varphi^\epsilon(x, t) := \varphi(x, t) + \epsilon v^-\left(\frac{x}{\epsilon}; p_0\right)
\end{equation}
where $v^-$ is given by Proposition \ref{periodictheorem} for $\nu$ sufficiently small to be determined. We claim that for $r, \epsilon$ sufficiently small, the following holds in the viscosity sense:
\begin{equation} \label{periodichomogthing}
\varphi^\epsilon_t + \bar{J} - \int J(y) \exp\left(\frac{\varphi^\epsilon(x - \epsilon y, t) - \varphi^\epsilon(x, t)}{\epsilon}\right) dy - c\left(\frac{x}{\epsilon}\right) \geq \frac{\theta}{2} \text{ in } B(x_0, r) \times (t_0 - r, t_0 + r).
\end{equation}
To show (\ref{periodichomogthing}), select another smooth test function $\psi$ and a point $(x_1, t_1) \in B(x_0, r) \times (t_0 - r, t_0 + r)$ such that $(x, t) \mapsto (\varphi^\epsilon - \psi)(x, t)$ has a global minimum at $(x_1, t_1)$. If we define $\eta(z, t) := \epsilon^{-1}(\psi(\epsilon z, t) - \varphi(\epsilon z, t)),$
then we know that
$$
(z, t) \mapsto v^-(z) - \eta(z, t) \text{ has a global minimum at } \left(\frac{x_1}{\epsilon}, t_1\right).
$$
In particular, we know that $\varphi_t(x_1, t_1) = \psi_t(x_1, t_1)$, because $v$ doesn't depend on $t$. Now because $v^-$ is a viscosity solution of (\ref{approxcorrector2}), $\eta$ satisfies
$$
\bar{J} - \int J(y) \exp(-y\cdot p_0))\exp\left(\eta\left(\frac{x_1}{\epsilon} - y, t_1\right) - \eta\left(\frac{x_1}{\epsilon}, t_1\right)\right) dy - c\left(\frac{x_1}{\epsilon}\right) \geq \overline{H}(p_0) - \nu = -\varphi_t + \theta - \nu.
$$
We can write
$$\eta\left(\frac{x_1}{\epsilon} - y, t_1\right) - \eta\left(\frac{x_1}{\epsilon}, t_1\right) = \frac{\psi(x_1 - \epsilon y, t_1) - \psi(x_1, t_1)}{\epsilon} - \frac{\varphi(x_1 - \epsilon y, t_1) - \varphi(x_1, t_1)}{\epsilon},
$$
and as $\epsilon \rightarrow 0$, because $\varphi$ is smooth, $\epsilon^{-1} (\varphi(x_1 - \epsilon y, t_1) - \varphi(x_1, t_1)) \rightarrow -y \cdot D\varphi(x_1, t_1),$ so then for $r, \epsilon, \nu$ sufficiently small, we get
$$
\bar{J} -\int J(y) \exp\left(\frac{\psi(x_1 - \epsilon y, t_1) - \psi(x_1, t_1)}{\epsilon}\right) dy - c\left(\frac{x_1}{\epsilon}\right) \geq \overline{H}(p_0) - \frac{\theta}{2} = -\psi_t + \frac{\theta}{2},
$$
which is exactly (\ref{periodichomogthing}). Here we have used the fact that $v^-$ is Lipschitz continuous. On the other hand, applying \eqref{kppcondition} to  (\ref{phiequation}) 
yields that 
\begin{equation} \label{phiepeq}
\phi^\epsilon_t + \bar{J} - \int J(y) \exp\left(\frac{\phi^\epsilon(x - \epsilon y, t) - \phi^\epsilon(x, t)}{\epsilon}\right)dy - c\left(\frac{x}{\epsilon}\right) \leq 0
\end{equation}
holds in the viscosity sense on $\mathbb{R}^n \times (0, \infty)$. Now because we have (\ref{periodichomogthing}) and (\ref{phiepeq}), we can use comparison to conclude that for $\epsilon$ sufficiently small,
$$
\max_{\overline{B(x_0, r) \times [t_0 - r, t_0 + r]}}(\phi^\epsilon - \varphi^\epsilon) = \max_{D}(\phi^\epsilon - \varphi^\epsilon),
$$
where $D = \{B(x_0, r)^c \times [t_0 - r, t_0+r]\} \cup \{B(x_0, r) \times \{t_0 - r\}\}$. Upon taking $\epsilon \rightarrow 0$, this identity contradicts our initial assumption that $\phi^\ast - \varphi$ has a strict global maximum at $(x_0, t_0)$.

It now remains to verify the initial condition; that is, we would like to show that
\begin{equation} \label{initialcondition}
\phi^* = \left\{
\begin{array}{ll}
0 & \text{ on } G_0 \times \{0\} \\
-\infty & \text{ on } \mathbb{R}^n \backslash G_0 \times \{0\}.
\end{array} \right.
\end{equation}
This follows in the same way as in \cite{majda}. It is clear by (\ref{hopfcole}) that for $x_0 \in G_0$, $\phi^\epsilon(x_0, 0) \rightarrow 0$ as $\epsilon \rightarrow 0$, and so $\phi^* = 0$ on $G_0 \times \{0\}$. So it remains to show that $\phi^* = -\infty$ on $\mathbb{R}^n \backslash G_0 \times \{0\}$. First we prove a preliminary claim. Fix $\mu > 0$ and select a smooth function $\zeta$ satisfying $\zeta = 0 \text{ on } G_0, \zeta > 0 \text{ on } \mathbb{R}^n \backslash G_0, 0 \leq \zeta \leq 1.$ We claim that
\begin{equation} \label{noboundarycondition}
\min(\phi^*_t+\overline{H}(D\phi^*), \phi^* + \mu \zeta) \leq 0 \text{ on } \mathbb{R}^n \times \{0\}
\end{equation}
holds in the viscosity sense. Suppose that $\varphi$ is a smooth test function and $\phi^* - \varphi$ has a strict local maximum at some $(x_0, 0) \in \mathbb{R}^n \times \{0\}$. If $x_0 \in G_0$, then $\zeta(x_0) = 0$ and (\ref{noboundarycondition}) holds. Otherwise, suppose that $x_0 \in \mathbb{R}^n \backslash G_0$, and that $\phi^*(x_0, 0) > -\mu \zeta(x_0) > -\infty.$ By definition of $\phi^*$, there exist points $(x_\epsilon, t_\epsilon)$ such that $(x_\epsilon, t_\epsilon) \rightarrow (x_0, 0)$ and $\phi^\epsilon(x_\epsilon, t_\epsilon) \rightarrow \phi^*(x_0, 0)$, but because $\phi^\epsilon(x, 0) = -\infty$ for all $x$ near $x_0$, the points $(x_\epsilon, t_\epsilon)$ must lie in $\mathbb{R}^n \times (0, \infty)$, and so we can repeat the preceding homogenization argument to show that $\varphi^*_t + \overline{H}(D\varphi^*)\leq 0$ at $(x_0, 0)$, which gives us (\ref{noboundarycondition}).

Now take $x_0 \in \mathbb{R}^n \backslash G_0$, and suppose for a contradiction that $\phi^*(x_0, 0) > -\infty$. Fix $\delta > 0$, and define $\varphi^\delta(x, t) = \delta^{-1}|x-x_0|^2 + \gamma t,$ for $\gamma$ to be selected (in terms of $\delta$) below. Since $\phi^*$ is upper semicontinuous and bounded above, we know that $\phi^* - \varphi^\delta$ has a maximum at some point $(x_\delta, t_\delta) \in \mathbb{R}^n \times [0, \infty)$. This implies that
\begin{equation} \label{stringofinequalities}
-\delta^{-1}|x_\delta - x_0|^2 \geq \phi^*(x_\delta, t_\delta) - (\delta^{-1}|x_\delta-x_0|^2 + \gamma t_\delta) \geq \phi^*(x_0, 0) > -\infty.
\end{equation}
Now if $t_\delta > 0$, then we know that $\varphi^\delta_t(x_\delta, t_\delta) +\overline{H}(D\varphi^\delta(x_\delta, t_\delta)) \leq 0,$ which means that
\begin{equation}\label{contradiction2}
\gamma + \overline{H}\left(-\frac{2(x_\delta-x_0)}{\delta}\right) \leq 0,
\end{equation}
but this is a contradiction upon taking $\gamma = \gamma(\delta)$ sufficiently large by (\ref{stringofinequalities}). Therefore, $t_\delta = 0$. Now if $\phi^*(x_0, 0) > -\mu\zeta(x_0)$, then since $x_\delta \rightarrow x_0$ by (\ref{stringofinequalities}), then this means that $\phi^*(x_\delta, 0) > -\mu\zeta(x_\delta)$ for $\delta$ sufficiently small. Therefore, by \eqref{noboundarycondition}, we get \eqref{contradiction2}, which is once again a contradiction. Therefore, we must have that $\phi^*(x_0, 0) \leq -\mu\zeta(x_0)$. However, since $\zeta(x_0) > 0$ and $\mu$ is arbitrary, then (\ref{stringofinequalities}) cannot hold and we have another contradiction. This finishes our proof of \eqref{initialcondition}.

We can prove in a similar fashion that $\phi_*$ is a supersolution of (\ref{effectiveequation}), using $v^+$ instead of $v^-$ for the perturbed test function \eqref{perturbedtestfunction}; the proof differs at the point where we deduce an analogous statement to (\ref{phiepeq}). Due to the nature of the variational inequality \eqref{effectiveequation}, it suffices to prove that $\phi_*$ is a supersolution of \eqref{effectiveequation} on $\{\phi_* < 0\}$. In this case, instead of using (\ref{kppcondition}), we use the fact that $(u^\epsilon)^{-1} f(z, u^\epsilon) \rightarrow c(z) \text{ on } \{\phi_* < 0\},$ which follows from \ref{hopfcole}.
Note that if $\phi_* = 0$, then the preceding statement would not hold, \eqref{kppcondition} would not give the correct inequality to prove that $\phi_*$ is a supersolution of $\phi_t + \overline{H}(p) = 0$
if we attempted to duplicate the proof from the subsolution case. It is precisely at this point that the variational inequality \eqref{effectiveequation} for $\phi$ is necessary.

The proof also deviates from the subsolution case when we show that the initial condition (\ref{initialcondition}) holds for $\phi_*$. Because in this case we know that $\phi_* = -\infty \text{ on } (\mathbb{R}^n \backslash G_0) \times \{0\}$ since $\phi^\epsilon = -\infty$ on that set, we need to show that $\phi_* = 0$ on $G_0 \times \{0\}$. Instead of \eqref{noboundarycondition}, in this case we show that $\max(\phi_{*, t} - \overline{H}(D\phi_*), \phi_*) \geq 0 \text{ on } G_0 \times \{0\},$ and in the proof we change the definition of $\varphi^\delta$ to $\varphi^\delta = -\delta^{-1}|x - x_0|^2 - \gamma t.$
Because $\overline{H}$ satisfies (\ref{hbarcoercive}), the result of \cite{cls} can be applied, which means that comparison holds for (\ref{effectiveequation}), and so $\phi^* = \phi_* = \phi$. This implies that $\phi^\epsilon$ converges locally uniformly to $\phi$, which was what we wanted.
\end{proof}
We now discuss the properties of the effective Hamiltonian $\overline{H}$. In the case of a homogeneous nonlinearity $f$ i.e. $c(z) \equiv c$, constant functions are correctors, so we can write the form of $\overline{H}(p)$ to be
$$
\overline{H}(p) = \bar{J} -c-\int J(y) \exp(-y\cdot p) dy.
$$
In particular, we can see that in this situation, the effective Hamiltonian is concave, negatively coercive, and continuous in $p$, and we now prove that these properties of $\overline{H}(p)$ hold in general.
\begin{proposition} \label{Hbarproperties}
The effective Hamiltonian $\overline{H}$ has the following properties:
\begin{compactenum}
\item $p \mapsto \overline{H}(p)$ is concave.
\item There exist positive constants $K_1, K_2, K_3, K_4> 0, C_1, C_2$ such that $\overline{H}(p)$ satisfies
\begin{equation} \label{hbarcoercive}
-K_1 \exp(K_2|p|) - C_1 \geq \overline{H}(p) \geq -K_3 \exp(K_4|p|) - C_2
\end{equation}
for all $p \in \mathbb{R}^n$. In particular, this implies that $\overline{H}$ is uniformly and negatively coercive.
\item There exist constants $C_3, C_4$ such that for all $p_1, p_2 \in \mathbb{R}^n$,
\begin{equation} \label{hbarloclipschitz}
|\overline{H}(p_1) - \overline{H}(p_2)| \leq C_3\exp(C_4(1+|p_1|+|p_2|))|p_1 - p_2|.
\end{equation}
\end{compactenum}
\end{proposition}
\begin{proof}
Select $p_1, p_2 \in \mathbb{R}^d$, and for each $\lambda > 0$, write $v^\lambda_i := v^\lambda(z, \omega; p_i)$. Set $p := \frac{p_1 + p_2}{2}$ and $\tilde{v}^\lambda := \frac{1}{2}(v^\lambda_1 + v^\lambda_2)$. It is clear that by convexity of the exponential we have that $\tilde{v}^\lambda$ satisfies
$$
\lambda \tilde{v}^\lambda + \bar{J} -\int J(y) \exp(-y\cdot p) \exp(\tilde{v}^\lambda(z-y) - \tilde{v}^\lambda(z)) dy - c(z) \geq 0.
$$
By comparison we have that $v^\lambda(y; p) \leq \tilde{v}^\lambda(y)$. Multiplying this inequality by $-\lambda$ and taking $\lambda \rightarrow 0$, we obtain
$\overline{H}(p) \geq \frac{1}{2}(\overline{H}(p_1) + \overline{H}(p_2)),$ which means that $\overline{H}$ is concave.

To prove (\ref{hbarcoercive}), we first note that since $c(z)$ is bounded, we can find sufficiently large $K_3, K_4, C_2$ so that $-\lambda^{-1}(K_3 \exp(K_4|p|) + C_2)$ is a subsolution of (\ref{approxcellproblem}). This gives half of (\ref{hbarcoercive}). To prove the other half, we note that since $J$ is symmetric, there exists $K_1, K_2$ such that
$$
\int J(y) \exp(-y \cdot p) dy \geq K_1 \exp(K_2|p|),
$$
and so this means that $-\lambda^{-1}(K_1 \exp(K_2|p|) - C_1)$ is a supersolution of (\ref{approxcellproblem}). This gives us the other half of (\ref{hbarcoercive}).

To show (\ref{hbarloclipschitz}), we prove a lemma giving a modulus of continuity estimate for $v^\lambda$.
\begin{lemma} \label{plipschitzlemma}
There exist $C_3, C_4 > 0$ such that for each $\lambda > 0$, $p_1, p_2 \in \mathbb{R}^n$
\begin{equation} \label{plipschitz}
\sup_{y \in \mathbb{R}^n} |\lambda v^\lambda(y; p_1) - \lambda v^\lambda(y; p_2)| \leq C_3\exp(C_4(1+|p_1|+|p_2|))|p_1 - p_2|.
\end{equation}
\end{lemma}
\begin{proof}\let\qed\relax
Let $v_1^\lambda, v_2^\lambda$ be the solution of (\ref{approxcellproblem}) with $p_1$ and $p_2$ respectively. We claim that there exist constants $C_3, C_4$ such that $v_1^\lambda \pm C_3 \lambda^{-1} \exp(C_4(1+|p_1|+|p_2|))|p_1 - p_2|$ are respectively a supersolution and a subsolution of (\ref{approxcellproblem}) with $p_2$, which by comparison for (\ref{approxcellproblem}) implies (\ref{plipschitz}). As the other case follows similarly, we show here that $$\tilde{v} := v_1^\lambda + C_3 \lambda^{-1} \exp(C_4(1+|p_1|+|p_2|))|p_1 - p_2|$$
is a supersolution of (\ref{approxcellproblem}) with $p_2$. Define $A := C_3 \exp(C_4(1+|p_1|+|p_2|))|p_1 - p_2|$ for notational convenience. We can compute
\begin{align}
\begin{split} \label{lipschitzproof}
&\lambda \tilde{v}(z) + \bar{J} - \int J(y) \exp(-y \cdot p_2) \exp(\tilde{v}(z-y) - \tilde{v}(z))dy -c(z) \\
&= \lambda v_1^\lambda(z) + A + \bar{J} - \int J(y) \exp(-y \cdot p_2) \exp(v_1^\lambda(z-y) - v_1^\lambda(z))dy -c(z) \\
&= \lambda v_1^\lambda(z) + A + \bar{J} - \int J(y) \exp(-y \cdot (p_1+ (p_2 - p_1)) \exp(v_1^\lambda(z-y) - v_1^\lambda(z))dy -c(z) \\
& \geq \lambda v_1^\lambda(z) + A  + \bar{J} - \int J(y) \exp(-y \cdot p_1) \exp(v_1^\lambda(z-y) - v_1^\lambda(z))dy -c(z) \\
&+ (1-\exp(K_1|p_1 - p_2|))\int J(y) \exp(-y \cdot p_1) \exp(v_1^\lambda(z-y) - v_1^\lambda(z))dy,
\end{split}
\end{align}
where $K_1$ is a constant. Now we know that because $v^\lambda_1$ is a solution of \eqref{approxcellproblem}, the quantity
$$
\int J(y) \exp(-y \cdot p_1) \exp(v_1^\lambda(z-y) - v_1^\lambda(z))dy
$$
is uniformly bounded in $\lambda$. Therefore, if we take appropriate constants $C_3, C_4$, then the last term of \eqref{lipschitzproof} is nonnegative, which shows that $\tilde{v}$ is a supersolution of \eqref{approxcellproblem} with $p_2$, as desired.
\end{proof}
\eqref{hbarloclipschitz} now follows from Lemma \ref{plipschitzlemma} by taking $\lambda \rightarrow 0$ in \eqref{plipschitz}.
\end{proof}
Note that these properties of $\overline{H}$ imply that (\ref{effectiveequation}) has a unique solution (see \cite{usersguide}, \cite{cls}).
\section{Convergence of $u^\epsilon$}
\label{sec:uepsilonconvergence}
We will now prove Theorem \ref{maintheorem}. That is, we show that as $\epsilon \rightarrow 0$, $u^\epsilon$, the solution of (\ref{uequation}), converges locally uniformly to 0 and 1 in regions determined by $\phi$, the solution of (\ref{effectiveequation}). This proof is based on \cite{majda} and \cite{souganidisevans}.
\begin{proof}[Proof of Theorem \ref{maintheorem}]
We have shown that
\begin{equation}\label{locunifconvergence}
\phi^\epsilon \rightarrow \phi \text{ locally uniformly},
\end{equation}
and this combined with the Hopf-Cole transformation (\ref{hopfcole}) implies $u^\epsilon = \exp(\epsilon^{-1}\phi^\epsilon) = \exp(\epsilon^{-1}(\phi + o_\epsilon(1))),$
so $u^\epsilon \rightarrow 0$ locally uniformly on $\{\phi < 0\}$. It remains to show that $u^\epsilon \rightarrow 1$ locally uniformly on $\mathrm{int}\{\phi = 0\}$. Fix a point $(x_0, t_0) \in \mathrm{int}\{\phi = 0\}$, and define $\varphi(x, t) = -|x- x_0|^2 - |t - t_0|^2.$ Then if we consider a domain $B(x_0, r) \times [t_0 - h, t_0 + h] \subset \mathrm{int}\{\phi = 0\}$ for $r, h$ sufficiently small, and define $(x_\epsilon, t_\epsilon)$ to be a point where $\phi^\epsilon - \varphi$ is minimized over this domain, (\ref{locunifconvergence}) implies that $(x_\epsilon, t_\epsilon) \rightarrow (x_0, t_0)$ as $\epsilon \rightarrow 0$. We now use $\varphi$ as a test function in (\ref{phiequation}), which we can do for $\epsilon$ sufficiently small,
to see that at $(x_\epsilon, t_\epsilon)$,
$$
\varphi_t(x_\epsilon, t_\epsilon) + \bar{J} -\int J(y) \exp\left(\frac{\varphi(x_\epsilon - \epsilon y, t_\epsilon) - \varphi(x_\epsilon, t_\epsilon)}{\epsilon}\right) dy \geq (u^\epsilon)^{-1} f(z_\epsilon, u^\epsilon),
$$
where $z_\epsilon := \epsilon^{-1}x_\epsilon.$ We have that
$$
\int J(y) \exp\left(\frac{\varphi(x_\epsilon - \epsilon y, t_\epsilon) - \varphi(x_\epsilon, t_\epsilon)}{\epsilon}\right)dy = \int J(y) \exp(-2y \cdot (x_\epsilon-x_0) + \epsilon |y|^2)dy  = \bar{J} + o_\epsilon(1)
$$
as $\epsilon \rightarrow 0$. Therefore, we can conclude that as $\epsilon \rightarrow 0$,
\begin{equation} \label{finequality3}
f(z_\epsilon, u^\epsilon(x_\epsilon, t_\epsilon)) \leq o(1) u^\epsilon(x_\epsilon, t_\epsilon),
\end{equation}
but we also have by (\ref{kppcondition}) that there exists $a > 0$ such that uniformly for all $z$ and all $U \in [0, 1]$,
\begin{equation} \label{finequality1}
f(z, U) \geq -a U^2 + c(z) U \geq -a U^2 + \kappa U.
\end{equation}
In particular by \eqref{fassumption} we can take
$$
a = \frac{1}{2} \sup_{x \in \mathbb{R}^n, u \in [0, 1]} |f_{uu}(x, u)|,
$$
and then if we consider $\underbar{\emph{f}}(x, u) = f(x, u) + a u^2 - \kappa u$, we have by \eqref{kppcondition} that $\underbar{\emph{f}}(x, 0) = 0, \underbar{\emph{f}}_u(x, 0) \geq 0,$ for all $x$, and $\underbar{\emph{f}}_{uu}(x, u) \geq 0$ for $u \in [0, 1]$ and all $x$. Therefore $\underbar{\emph{f}}(x, u) \geq 0$ for $u \in [0, 1]$ and all $x$, and thus \eqref{finequality1} is proved.

\eqref{finequality3} and \eqref{finequality1} imply that $o_\epsilon(1) u^\epsilon(x_\epsilon, t_\epsilon) \geq -a (u^\epsilon(x_\epsilon, t_\epsilon))^2 + \kappa u^\epsilon(x_\epsilon, t_\epsilon).$ Because $u^\epsilon$ is nonnegative by Lemma \ref{kpplemma}, $o_\epsilon(1) \geq -a u^\epsilon(x_\epsilon, t_\epsilon) + \kappa,$ and so
\begin{equation} \label{blah9}
\liminf_{\epsilon \rightarrow 0} u^\epsilon(x_\epsilon, t_\epsilon) \geq \frac{\kappa}{a} > 0.
\end{equation}
However, since $(x_\epsilon, t_\epsilon)$ is a local minimum of $\phi^\epsilon - \varphi$, we have $\phi^\epsilon(x_\epsilon, t_\epsilon) - \varphi(x_\epsilon, t_\epsilon) \leq \phi^\epsilon(x_0, t_0) - \varphi(x_0, t_0)$
which means that $\epsilon \log(u^\epsilon(x_\epsilon, t_\epsilon)) - o_\epsilon(1) \leq \epsilon \log(u^\epsilon(x_0, t_0)),$ and so due to \eqref{blah9}, $u^\epsilon(x_0, t_0) \geq u^\epsilon(x_\epsilon, t_\epsilon) +o_\epsilon(1) \geq \frac{\kappa}{a} > 0 \text{ as } \epsilon \rightarrow 0,$
Therefore,
\begin{equation} \label{uepsilonpositive}
\liminf_{\epsilon \rightarrow 0} u^\epsilon \geq \alpha > 0
\end{equation}
uniformly on any compact set $O \subset \mathrm{int}\{\phi = 0\}$ for $\alpha = \alpha(O)$.

We now need to show that
\begin{equation} \label{convergenceto1}
\displaystyle \lim_{\epsilon \rightarrow 0} u^\epsilon = 1
\end{equation}
uniformly on compact subsets of $\mathrm{int}\{\phi = 0\}$. First note that it suffices to consider the cylinder $A = B(x_0, r_0) \times (t_0, t_0 +h) \subset \mathrm{int}\{\phi = 0\}$ and to prove \eqref{convergenceto1} in $A' = \overline{B(x_0, \frac{r_0}{2})} \times (t_0+ \frac{h}{2}, t_0 + h)$. We know that \eqref{uepsilonpositive} holds uniformly on $A'$. Due to (\ref{fassumption}), (\ref{monostable}), and (\ref{kppcondition}), for any $\theta > 0$ we can find a sufficiently small $\beta = \beta(\alpha, \theta) > 0$ such that
\begin{equation} \label{finequality2}
f(z, u) \geq \beta(1-\theta-u)
\end{equation}
for all $u \in [\frac{\alpha}{2}, 1]$ and for all $z \in \overline{B(x_0, r_0)}$. We can prove this by noting that because $f \geq 0$ for all $u \in [0, 1]$, for any $\theta, \beta > 0$ we know that $\tilde{f}(z, u) := \frac{f(z, u)}{\beta} - (1-\theta) + u \geq 0$ for all $z \in \overline{B(x_0, r_0)}$ and $u \in [1-\theta, 1]$. Therefore, we need to show that \eqref{finequality2} holds for $u \in [\frac{\alpha}{2}, 1-\theta]$. Then, because $z$ lies in a compact set, then by \eqref{monostable} there exists $\rho > 0$ such that $f(z, u) \geq \rho$ for all $(z, u) \in \overline{B(x_0, r_0)} \times [\frac{\alpha}{2}, 1-\theta]$. If we now take $\beta = \rho$, it is clear that \eqref{finequality2} holds for $(z, u) \in \overline{B(x_0, r_0)} \times [\frac{\alpha}{2}, 1-\theta]$, thus giving us the inequality we need.

Applying \eqref{finequality2} to (\ref{uequation}) gives us
\begin{equation} \label{blah6}
u^\epsilon_t - \frac{1}{\epsilon} \int J(y) \left(u^\epsilon(x - \epsilon y, t)  - u^\epsilon(x, t)\right)dy + \frac{\beta}{\epsilon}(u^\epsilon - 1 + \theta) \geq 0 \textrm{ in } A.
\end{equation}
Now define
\begin{equation*}
f^\epsilon(t) = 1 - \theta - e^{-\frac{\delta(t-t_0)}{\epsilon}}, g^\epsilon(x, t) = f^\epsilon(t) \psi(x) - \epsilon \gamma (t - t_0),
\end{equation*}
where $\psi$ satisfies $0 \leq \psi \leq 1, \psi = 0 \text{ on } B(x_0, r_0)^c, \psi = 1 \text{ on } B\left(x_0, \frac{r_0}{2}\right).$ Using the fact that $J$ is symmetric, we have that as $\epsilon \rightarrow 0$,
\begin{align*}
\frac{1}{\epsilon} \int J(y) (g^\epsilon(x - \epsilon y, t) - g^\epsilon(x, t)) dy &= \frac{\epsilon}{2} \int J(y) \left(\frac{g^\epsilon(x+\epsilon y, t) - 2g^\epsilon(x, t) + g^\epsilon(x - \epsilon y, t)}{\epsilon^2}\right) dy \\
&\geq -C\epsilon \int J(y) |D^2 g^\epsilon(x, t)| dy \geq -C\epsilon
\end{align*}
and we can compute
\begin{align*}
g^\epsilon_t(x, t) + \frac{\beta}{\epsilon}(g^\epsilon(x, t) - 1 + \theta) &= -\epsilon \gamma + \frac{\delta}{\epsilon}e^{-\frac{\delta(t-t_0)}{\epsilon}} \psi(x) + \frac{\beta}{\epsilon}(f^\epsilon(t) \psi(x) - \epsilon \gamma(t - t_0) - 1 + \theta) \\
&\leq \frac{\delta}{\epsilon}[1-\theta-f^\epsilon(t) \psi(x)] + \frac{\beta}{\epsilon}(f^\epsilon(t) \psi(x) - 1 + \theta) \\
&\leq -\epsilon \gamma + \frac{\beta - \delta}{\epsilon}(f^\epsilon(t) \psi(x) - 1 + \theta),
\end{align*}
and so
$$
g^\epsilon_t -\frac{1}{\epsilon} \int J(y) (g^\epsilon(x - \epsilon y, t) - g^\epsilon(x, t)) dy + \frac{\beta}{\epsilon}(g^\epsilon - 1 + \theta) \leq 0 \text{ in } A,
$$
if we take $\gamma > 0$ sufficiently large and $\delta >0$ sufficiently small. This means that $g^\epsilon$ is a subsolution of (\ref{blah6}) in $A$. We have by construction that $g^\epsilon \leq 0$ on $(\mathbb{R}^n \times \{t_0\}) \cup (B(x_0, r_0)^c \times (t_0, t_0 + h))$, so by comparison we have that $g^\epsilon \leq u^\epsilon$ in $A$, but this means that since $\lim_{\epsilon \rightarrow 0} g^\epsilon = 1 - \theta \text{ in } A'$ and $\theta$ is arbitrary, we have \eqref{convergenceto1}. This finishes the proof of Theorem \ref{maintheorem}.
\end{proof}
\section*{Acknowledgements}
I'd like to thank Panagiotis Souganidis for suggesting this problem to me and to thank Panagiotis Souganidis and Benjamin Fehrman for many useful conversations. In addition I'd like to thank Panagiotis Souganidis and Luis Silvestre for their numerous suggestions and advice throughout the process of writing and editing this paper.

This material is based upon work supported by the National Science Foundation Graduate Research Fellowship under Grant No. DGE-1144082.
\bibliography{FullPaper}
\bibliographystyle{abbrv}
\end{document}